\documentclass[oneside,reqno]{amsart}
\usepackage{amssymb}
\usepackage{amsmath}
\usepackage{amsthm}
\usepackage{enumerate}
\usepackage{graphicx}
\usepackage{amsfonts}

\numberwithin{equation}{section}

\newcommand{\N}{\mathbb N}

\newtheorem{theorem}{Theorem}[section]

\newtheorem{lem}{Lemma}[section]
\newtheorem{rem}{Remark}[section]

\begin{document}
\title{Rate of convergence of Gupta-Srivastava operators based on certain parameters}
\author{Ram Pratap and Naokant Deo}
\date{}
\address{Delhi Technological University\newline
\indent Formerly Delhi College of Engineering\newline
\indent Department of Applied Mathematics\newline
\indent Bawana Road, 110042 Delhi, India}
\email{rampratapiitr@gmail.com}
\email{naokantdeo@dce.ac.in}
\begin{abstract} 
In the present paper, we consider the B\'ezier variant of the general family of Gupta-Srivastava operators \cite{GS:18}. For the proposed operators, we discuss the rate of convergence by using of Lipschitz type space, Ditzian-Totik modulus of smoothness, weighted modulus of continuity and functions of bounded variation.\end{abstract}
\maketitle\markboth{Ram Pratap and Naokant Deo}{Gupta-Srivastava Approximation Operators}

\section{introduction}
In the year $2003$,  Srivastava et al. \cite{SG:03} introduced a sequence of positive linear operators and studied the convergence properties. After two year Ispir and Y\"{u}ksel proposed the B\'ezier variant of Srivastava-Gupta operators  and discussed the rate of convergence for the proposed operators  \cite{IY:05}. Several researchers have taken in to out with different generalised form of the Srivastava-Gupta operators (see \cite{TMM:15}, \cite{Deo:12}, \cite{KA:17}, \cite{MP:15}, \cite{VA:12}, \cite{Yadav:14}). One of them, Yadav \cite{Yadav:14} studied the modification of Srivastava-Gupta operators which preserve the constant functions as well as linear functions, established the Voronovskaya type theorem and statistical convergence. In $2017$, Neer et al. \cite{NIA:17} proposed the B\'ezier variant of the operators, which was introduced by Yadav \cite{Yadav:14} and discussed several convergence properties. \\

Recently, Gupta et al. \cite{GS:18} proposed a general family of a positive linear operator, which preserve constant functions as well as linear functions for all  $c \in \Bbb N \cup \left\{ 0\right\}\cup\left\{-1\right\}$. They have considered the general sequence of positive linear operators containing some well-known operators as special cases of Srivastav-Gupta operators~\cite{SG:03}, and  for $m \in \Bbb Z,$ and $c \in \Bbb N \cup \left\{ 0\right\}\cup\left\{-1\right\}$ operators defined as:
\begin{align}\label{L:1}
\nonumber{L_{n,m}^{c}}\left( {f\left( t \right);x} \right)& = \left\{ {n + (m + 1)c} \right\}\sum\limits_{k = 1}^\infty  {{p_{n + mc,k}}} (x;c)\int\limits_0^\infty  {{p_{n + (m + 2)c,k - 1}}(t;c)} f(t)dt\\
 &\hspace{3cm}+ {p_{n + mc,0}}(x;c)f(0),
\end{align}
where \[{p_{n,k}}\left( {x;c} \right) = \frac{{{{\left( { - x} \right)}^k}}}{{k!}}\phi _{n,c}^{(k)}(x),\]
and 
\[\phi _{n,c}^{(k)} = \left\{ {\begin{array}{*{20}{l}}
  {{e^{ - nx}},\qquad \;\;\;\;\;\; c = 0} \\ 
  {{{(1 - x)}^{ - n}},\;\;\;\;\;\; c =  - 1} \\ 
  {{{(1 + cx)}^{\frac{{ - n}}{c}}},\;\;\;\;c = 1,2,3,...} 
\end{array}} \right.\]
Inspired from the above works, depending upon some parameter $\alpha \ge 1$ we propose here the following sequence of operators (also called B\'ezier variant) (\ref{L:1}) as follows:
\begin{align}\label{L:2}
\nonumber F_{n,m}^{c,\alpha}\left( {f\left( t \right);x} \right) &= \left\{ {n + (m + 1)c} \right\}\sum\limits_{k = 1}^\infty  {Q_{n + mc}^{(\alpha )}} (x;c)\int\limits_0^\infty  {{p_{n + (m + 2)c,k - 1}}(t;c)} f(t)dt\\
&\hspace{1.5cm}+ Q_{n + mc,0}^{(\alpha )}(x;c)f(0),
\end{align}
where $Q_{n + mc,k}^{(\alpha )}(x;c) = {\left( {{J_{n + mc,k}}(x,c)} \right)^\alpha } - {\left( {{J_{n + mc,k + 1}}(x,c)} \right)^\alpha }$, $\alpha\ge1$ with
\[{J_{n + mc,k}}(x,c) = \sum\limits_{j = k}^\infty  {{p_{n + mc,k}}(x,c)} ,\]
where $k< \infty$ and otherwise zero. It is obvious, the operators $F_{n,m}^{c,\alpha}\left( {.;x} \right)$ are the linear positive operator. For $\alpha=1$ the operators (\ref{L:2}) reduce to (\ref{L:1}).\\
The special cases of operators (\ref{L:2}) are given below:
\begin{itemize}
\item  [(i)] For $c=0$, $\alpha=1$ and ${\phi _{n,0}}(x) = {e^{ - nx}}$, we get Phillips operators  
\begin{align*}
L_{n,m}^0\left( {f\left( t \right);x} \right) &= n\sum\limits_{k = 1}^\infty  {{p_{n + mc,k}}} (x;0)\int\limits_0^\infty  {{p_{n + mc,k - 1}}(t;0)} f(t)dt\\
&\hspace{1.5cm}+ {p_{n,0}}(x;0)f(0),
\end{align*}	
where $${p_{n,k}}(x,0) = \frac{{{e^{ - nx}}{{(nx)}^k}}}{{k!}}\;\;\text{and}\;\;x\in[0,\infty).$$
\item  [(ii)] For $c \in \Bbb N$, $\alpha=1$ and ${\phi _{n,c}}(x) = {(1 + cx)^{ - \frac{n}{c}}}$, we get genuine Baskakov-Durrmeyer type operators. These operators are similar to (\ref{L:1}) except for $c=\{0,-1\}$, called summation integral type of operators,
where $${p_{n,k}}(x;c) = \frac{{{{\left( {\frac{n}{c}} \right)}_k}}}{{k!}}\frac{{{{\left( {cx} \right)}_k}}}{{{{\left( {1 + cx} \right)}^{\frac{n}{c} + k}}}},$$ and ${(n)_i}$ denotes the rising factorial given by \[{(n)_i} = n(n + 1)(n + 2)...(n + i - 1)\;\;\&\;\;(n)_{0}=1(i \in \Bbb N).\]
\item  [(iii)] For $c=-1$, $\alpha=1$ and ${\phi _{n,-1}}(x)=(1-x)^{-n}$, we have a sequence of Bernstein-Durrmeyer operators
\begin{align}
\nonumber  {L_{n,m}^{-1}}\left( {f,x} \right) &= \left( {n - m - 1} \right)\sum\limits_{k = 1}^{n - m - 1} {{p_{n - m,k}}\left( {x, - 1} \right)} \int\limits_0^1 {{p_{n - m - 2,k - 1}}\left( {t, - 1} \right)f(t)dt} \\
&\hspace{1.5cm} + {p_{n - m,0}}\left( {x, - 1} \right)f(0) + {p_{n - m,n - m}}\left( {x, - 1} \right),
\end{align}
where $${p_{n,k}}(x; - 1) = \left( \begin{array}{l}
n\\
k
\end{array} \right){x^k}{(1 - x)^{n - k}}.$$
\end{itemize}
The purpose of this article is to investigate the approximation results by  using of Lipchitz type space, Ditzian-Totik modulus of smoothness, weighted modulus of continuity and functions of bounded variation.

In the year $2008$, Deo et al. an interesting modification of introduced a modified Bernstein operators

\section{Auxiliary Results}
In this section, we give some auxiliary results and with help of these results we study our main results.\\

Let $C[0,\infty)$ denotes the space of all continuous functions in $[0,\infty)$ and let $C_B[0,\infty)$ be the space of all continuous and bounded functions in $[0,\infty)$.
\begin{lem}\label{lem-1}
Let $f(t)=e_{i}=t^{i}$, $i=0,1,2,3,4,\;\;\;c \in \Bbb N \cup \left\{ 0 \right\}\cup\left\{-1\right\}$ and $m \in \mathbb{Z} $, then we have
\begin{enumerate}
\item[(i)]${L_{n,m}^{c}}\left( {e_0;x} \right)= 1$;
\item[(ii)]${L_{n,m}^{c}}\left( {e_1;x} \right)= x$;
\item[(iii)] ${L_{n,m}^{c}}\left( {{e_2};x} \right)= \frac{{(n + (m + 1)c)}}{{(n + (m - 1)c)}}{x^2} + \frac{2}{{(n + (m - 1)c)}}x$;
\item[(iv)] ${L_{n,c}}\left( {{e_3};x} \right) = \frac{{(n + (m + 1)c)(n + (m + 2)c)}}{{(n + (m - 1)c)(n + (m - 2)c)}}{x^3}$\\ $+ \frac{{6(n + (m + 1)c)}}{{(n + (m - 1)c)(n + (m - 2)c)}}{x^2} + \frac{6}{{(n + (m - 1)c)(n + (m - 2)c)}}x;$
\item[(v)] ${L_{n,c}}\left( {{e_4};x} \right) = \frac{{(n + (m + 1)c)(n + (m + 2)c)(n + (m + 3)c)}}{{(n + (m - 1)c)(n + (m - 2)c)(n + (m - 3)c)}}{x^4}$\\
$ + \frac{{12(n + (m + 1)c)(n + (m + 2)c)}}{{(n + (m - 1)c)(n + (m - 2)c)(n + (m - 3)c)}}{x^3} + \frac{{36(n + (m + 1)c)}}{{(n + (m - 1)c)(n + (m - 2)c)(n + (m - 3)c)}}{x^2}$\\
$ + \frac{{24}}{{(n + (m - 1)c)(n + (m - 2)c)(n + (m - 3)c)}}x$.
\end{enumerate}
All the moments of operators $(\ref{L:1})$  can be obtained in terms of hyper geometric function of order $r\in\N$ for details see~\cite{GS:18}.
\end{lem}

\begin{lem}\label{lm1}
The central moment of the operators (\ref{L:1}) is given as:
 \[{\mu _{n,s}^{c}}(x) = {L_{n,c}}({(t - x)^s};x),\]
 for $s= 2, 4 $ then, we have
\begin{enumerate}
\item[(ii)] ${\mu _{n,2}^{c}}(x) = \frac{{2x(1 + cx)}}{{(n + (m - 1)c)}};$
\item[(iii)]${\mu _{n,4}^{c}}(x) = \frac{{12{c^2}(n + (m + 7)c)}}{{(n + (m - 1)c)(n + (m - 2)c)(n + (m - 3)c)}}{x^4} + \frac{{24{c^2}(13n + (13m + 1)c)}}{{(n + (m - 1)c)(n + (m - 2)c)(n + (m - 3)c)}}{x^3}$\\
 $+ \frac{{12(n + (m + 9)c)}}{{(n + (m - 1)c)(n + (m - 2)c)(n + (m - 3)c)}}{x^2} + \frac{{24}}{{(n + (m - 1)c)(n + (m - 2)c)(n + (m - 3)c)}}x.$
\end{enumerate}
\end{lem}
\begin{rem}\label{rem-1}
If $n$ is sufficiently large then the central moment of the operators (\ref{L:1}) are:
 $${\mu _{n,2}^{c}}(x) \le C\frac{{x(1 + cx)}}{{n }},$$
 and 
 \[\mu _{n,4}^c(x) \le C\frac{{{{(x(1 + cx))}^2}}}{{{n^2}}},\]
where $C>0$ is constant.
\end{rem}
\begin{rem}\label{rem-2}
We know that $\sum\limits_{j = 0}^\infty  {{p_{n + mc,j}}(x,c)}  = 1$ and from (\ref{L:2}), we have 
\begin{align}\label{E:2}
\nonumber F_{n,m}^{c, \alpha }\left( {1;x} \right)& = \left\{ {n + (m + 1)c} \right\}\sum\limits_{k = 1}^\infty  {Q_{n + mc}^{(\alpha )}} (x;c)\int\limits_0^\infty  {{p_{n + (m + 2)c,k - 1}}(t;c)} dt\\ 
\nonumber &\hspace{2cm}+ Q_{n + mc,0}^{(\alpha )}(x;c)\\
\nonumber& = \sum\limits_{k = 0}^\infty  {Q_{n + mc}^{(\alpha )}} (x;c) = {\left( {{J_{n+mc,0}}(x,c)} \right)^{(\alpha )}}\\
\nonumber&= {\left( {\sum\limits_{k = 0}^\infty  {{p_{n + mc,j}}(x,c)} } \right)^{(\alpha )}} = 1.
\end{align}
\end{rem}
\begin{lem}\label{lem-2}
For each $f\in C_B[0,\infty)$ then, we have
\[\left| {F_{n,m}^{c, \alpha}\left( {f(t);x} \right)} \right| \le \left\| f \right\|.\]
\end{lem}

\begin{proof}
It is easy to prove the above result by using Remark \ref{rem-2}, therefore we skip the proof.
\end{proof}
\begin{lem}\label{lem-3}
For every $f\in C_B[0,\infty)$ then, we have
\[\left| {F_{n,m}^{c, \alpha}\left( {f(t);x} \right)} \right| \le \alpha {L_{n,m}^{c}}\left( {\left\| f \right\|;x} \right).\]
\end{lem}
\begin{proof}
For $0\le c\le d\le1$, $\alpha\ge 1$, using the inequality 
\[\left| {{c^\alpha } - {d^\alpha }} \right| \le \alpha \left| {c - d} \right|,\]
from the definition of $Q_{n + mc,k}^{(\alpha )}(x;c)$, for all $k \in \Bbb N \cup \left\{ 0\right\}$, we get 
\begin{align*}
0 &< {\left( {{J_{n + mc,k}}(x,c)} \right)^\alpha } - {\left( {{J_{n + mc,k + 1}}(x,c)} \right)^\alpha }\\
& \le \alpha \left( {{J_{n + mc,k}}(x,c) - {J_{n + mc,k + 1}}(x,c)} \right)\\
&= \alpha {p_{n + mc}}(x,c).
\end{align*}
Hence
\[\left| {F_{n,m}^{c,\alpha}\left( {f(t);x} \right)} \right| \le \alpha {L_{n,m}^{c}}\left( {\left\| f \right\|;x} \right).\]
\end{proof}
\section{Main Results}
For $x>0$, $t\ge 0$ and $0<\gamma \le 1$, we can see in \"{O}zarslan et al. \cite{OD:10}, the Lipschitz type space is defined as:
\[Li{p_M}(\gamma ): = \left\{ {f \in C[0,\infty ):\left| {f(t) - f(x)} \right| \le M\frac{{{{\left| {t - x} \right|}^\gamma }}}{{{{(t + x)}^{{\raise0.7ex\hbox{$\gamma $} \!\mathord{\left/
 {\vphantom {\gamma  2}}\right.\kern-\nulldelimiterspace}
\!\lower0.7ex\hbox{$2$}}}}}}} \right\}.\]
Now we estimate the rate of convergence of the function $f \in Li{p_M}(\gamma )$ by the operators $F_{n,m}^{c,\alpha }\left( {.;x} \right)$. 
\begin{theorem}\label{thm-1}
For $f \in Li{p_M}(\gamma )$ and $\gamma \in (0,1]$. Then for $x>0$, we have
 	\[\left| {F_{n,m}^{c,\alpha}\left( {f(t);x} \right) - f(x)} \right| \le \alpha M{\left( {\frac{{\delta _{n,m}^c(x)}}{x}} \right)^{{\raise0.7ex\hbox{$\gamma $} \!\mathord{\left/
 {\vphantom {\gamma  2}}\right.\kern-\nulldelimiterspace}
\!\lower0.7ex\hbox{$2$}}}},\]
where $\delta _{n,m}^c(x) = \sqrt {\frac{{2x(1 + cx)}}{{(n + (m - 1)c)}}}$. 	
\end{theorem}
\begin{proof}
Using Lemma \ref{lem-3}, we get 
\begin{align}\label{l1}
\nonumber \left| {F_{n,m}^{c,\alpha}\left( {f(t);x} \right) - f(x)} \right| &\le F_{n,m}^{c,\alpha}\left( {\left| {f(t) - f(x)} \right|;x} \right)\\
\nonumber &\le \alpha {L_{n,m}^{c}}\left( {\left| {f(t) - f(x)} \right|;x} \right)\\
\nonumber &\le \alpha M{L_{n,m}^{c}}\left( {\frac{{{{\left| {t - x} \right|}^\gamma }}}{{{{(t + x)}^{{\raise0.7ex\hbox{$\gamma $} \!\mathord{\left/
 {\vphantom {\gamma  2}}\right.\kern-\nulldelimiterspace}
\!\lower0.7ex\hbox{$2$}}}}}};x} \right)\\
&\le \frac{\alpha M}{{{x^{{\raise0.7ex\hbox{$\gamma $} \!\mathord{\left/
 {\vphantom {\gamma  2}}\right.\kern-\nulldelimiterspace}
\!\lower0.7ex\hbox{$2$}}}}}}{L_{n,m}^{c}}\left( {{{\left| {t - x} \right|}^\gamma };x} \right).	
\end{align}
Using H$\ddot{\text{o}}$lder's inequality by taking $p=2/\gamma$ and $q=2/(2-\gamma)$, we get 
\begin{align}\label{l2}
\nonumber {L_{n,m}^{c}}\left( {{{\left| {t - x} \right|}^\gamma };x} \right) &\le {\left\{ {L_{n,m}^c\left( {{{(t - x)}^2};x} \right)} \right\}^{\frac{\gamma }{2}}}.{\left\{ {L_{n,m}^c\left( {{1^{\frac{2}{{\left( {2 - \gamma } \right)}}}};x} \right)} \right\}^{\frac{{\left( {2 - \gamma } \right)}}{2}}}\\
&\le {\left\{ {L_{n,m}^c\left( {{{(t - x)}^2};x} \right)} \right\}^{\frac{\gamma }{2}}} = {\left( {\delta _{n,m}^c(x)} \right)^{\frac{\gamma }{2}}}.
\end{align}
From (\ref{l1}) and (\ref{l2}), we get
\[\left| {F_{n,m}^{c,\alpha }\left( {f(t);x} \right) - f(x)} \right| \leqslant \alpha M{\left( {\frac{{\delta _{n,m}^c(x)}}{x}} \right)^{\frac{\gamma }{2}}}.\]
Hence the proof.
\end{proof}

In our next theorem, we estimate the rate of convergence by using of Ditzian-Totik modulus of smoothness ${\omega _{{\phi ^\beta }}}(f,\delta )$ and Peetre $K-$functional ${K_{{\phi ^\beta }}}(f,\delta )$, $0\le \beta \le 1$. For $f\in C_B[0,\infty)$ and $\phi(x)=\sqrt{x(1+cx)}$, the Ditzian-Totik modulus of smoothness is defined as
\[{\omega _{{\phi ^\beta }}}(f,\delta ) = \mathop {\sup }\limits_{0 \le i \le \delta } \mathop {\sup }\limits_{x \pm \frac{{i{\phi ^\beta }(x)}}{2}\in [0,\infty)} \left| {f\left( {x + \frac{{i{\phi ^\beta }(x)}}{2}} \right) - f\left( {x - \frac{{i{\phi ^\beta }(x)}}{2}} \right)} \right|,\]
and the Peetre $K-$functional is defined as 
\[{\omega _{{\phi ^\beta }}}(f,\delta ) = \mathop {\inf }\limits_{g \in {W_\beta }} \left\{ {\left\| {f - g} \right\| - \delta \left\| {{\phi ^\beta }{g^{\prime}}} \right\|} \right\},\]
where $W_\beta$ is subspace of the space which is locally absolutely continuous functions $g$ on $[0,\infty)$, with the normed $\left\| {{\phi ^\beta }{g^{\prime}}} \right\|\le \infty$. In [\cite{DT:87}, Theorem 2.1.1], there exists a constant $C>0$ such that
\begin{equation}\label{l01}
{C^{ - 1}}{\omega _{{\phi ^\beta }}}(f,\delta ) \le {K_{{\phi ^\beta }}}(f,\delta ) \le C{\omega _{{\phi ^\beta }}}(f,\delta ).
\end{equation}
\begin{theorem}\label{thm-2}
For $f\in C_B[0,\infty)$ then, we have
	\[\left| {F_{n,m}^{c,\alpha}\left( {f(t);x} \right) - f(x)} \right| \le C{\omega _{{\phi ^\beta }}}\left( {f;\frac{{{\phi ^{1 - \beta }}(x)}}{{\sqrt n }}} \right),\]
	for sufficient large $n$ and $C$ is a positive constant independent from $f$ and $n$.
\begin{proof}
For $g\in W_\beta$, we get
\begin{equation}\label{e:01}
g(t) = g(x) + \int\limits_x^t {g^{\prime}} (u)du	.
\end{equation}
Applying $F_{n,m}^{c,\alpha }$ in (\ref{e:01}) and using H$\ddot{\text{o}}$lder's inequality then, we have
\begin{align}\label{l4}
\nonumber \left| {F_{n,m}^{c,\alpha}\left( {g(t);x} \right) - g(x)} \right| &
\le F_{n,m}^{c,\alpha }\left( {\int\limits_x^t {\left| {{g^{\prime}}} \right|} du;x} \right)\\
\nonumber &\le \left\| {{\phi ^\beta }{g^\prime }} \right\|F_{n,m}^{c,\alpha}\left( {\left| {\int\limits_x^t {\frac{{du}}{{{\phi ^\beta }(u)}}} } \right|;x} \right)\\
&\le \left\| {{\phi ^\beta }{g^\prime }} \right\|F_{n,m}^{c,\alpha }\left( {{{\left| {t - x} \right|}^{1 - \beta }}{{\left| {\int\limits_x^t {\frac{{du}}{{\phi (u)}}} } \right|}^\beta };x} \right).
\end{align}
Let us take $A = \left| {\int\limits_x^t {\frac{{du}}{{\phi (u)}}} } \right|$ then, we get	
\begin{align}\label{l5}
\nonumber A &\le \left| {\int\limits_x^t {\frac{{du}}{{\sqrt u }}} } \right|\left| {\left( {\frac{1}{{\sqrt {1 + cx} }} + \frac{1}{{\sqrt {1 + ct} }}} \right)} \right|\\
\nonumber &\le 2\left| {\sqrt t  - \sqrt x } \right|\left( {\frac{1}{{\sqrt {1 + cx} }} + \frac{1}{{\sqrt {1 + ct} }}} \right)\\
\nonumber &\le 2\frac{{\left| {t - x} \right|}}{{\sqrt x  + \sqrt t }}\left( {\frac{1}{{\sqrt {1 + cx} }} + \frac{1}{{\sqrt {1 + ct} }}} \right)\\
&\le 2\frac{{\left| {t - x} \right|}}{{\sqrt x }}\left( {\frac{1}{{\sqrt {1 + cx} }} + \frac{1}{{\sqrt {1 + ct} }}} \right).
\end{align}
The inequality ${\left| {a + b} \right|^\beta } \le {\left| a \right|^\beta } + {\left| b \right|^\beta }$, $0 \le \beta \le 1$ then from (\ref{l5}), we get
\begin{equation}\label{l6}
{\left| {\int\limits_x^t {\frac{{du}}{{\phi (u)}}} } \right|^\beta } \le 2^\beta\frac{{{{\left| {t - x} \right|}^\beta }}}{{{x^{{\raise0.7ex\hbox{$\beta $} \!\mathord{\left/
 {\vphantom {\beta  2}}\right.\kern-\nulldelimiterspace}
\!\lower0.7ex\hbox{$2$}}}}}}\left( {\frac{1}{{{{(1 + cx)}^{{\raise0.7ex\hbox{$\beta $} \!\mathord{\left/
 {\vphantom {\beta  2}}\right.\kern-\nulldelimiterspace}
\!\lower0.7ex\hbox{$2$}}}}}} + \frac{1}{{{{(1 + ct)}^{{\raise0.7ex\hbox{$\beta $} \!\mathord{\left/
 {\vphantom {\beta  2}}\right.\kern-\nulldelimiterspace}
\!\lower0.7ex\hbox{$2$}}}}}}} \right).
\end{equation}
From (\ref{l4}), (\ref{l6}) and using Cauchy inequality then, we get
\begin{align}\label{ç}
\nonumber \left| {F_{n,m}^{c,\alpha}\left( {g(t);x} \right) - g(x)} \right| \le & \frac{{{2^\beta }\left\| {{\phi ^\beta }{g^\prime }} \right\|}}{{{x^{{\raise0.7ex\hbox{$\beta $} \!\mathord{\left/
 {\vphantom {\beta  2}}\right.\kern-\nulldelimiterspace}
\!\lower0.7ex\hbox{$2$}}}}}}F_{n,m}^{c,\alpha}\left( {\left| {t - x} \right|\left( {\frac{1}{{{{(1 + cx)}^{{\raise0.7ex\hbox{$\beta $} \!\mathord{\left/
 {\vphantom {\beta  2}}\right.\kern-\nulldelimiterspace}
\!\lower0.7ex\hbox{$2$}}}}}} + \frac{1}{{{{(1 + ct)}^{{\raise0.7ex\hbox{$\beta $} \!\mathord{\left/
 {\vphantom {\beta  2}}\right.\kern-\nulldelimiterspace}
\!\lower0.7ex\hbox{$2$}}}}}}} \right);x} \right)	\\
\nonumber \le &\frac{{{2^\beta }\left\| {{\phi ^\beta }{g^\prime }} \right\|}}{{{x^{{\raise0.7ex\hbox{$\beta $} \!\mathord{\left/
 {\vphantom {\beta  2}}\right.\kern-\nulldelimiterspace}
\!\lower0.7ex\hbox{$2$}}}}}}\left( {\frac{1}{{{{(1 + cx)}^{{\raise0.7ex\hbox{$\beta $} \!\mathord{\left/
 {\vphantom {\beta  2}}\right.\kern-\nulldelimiterspace}
\!\lower0.7ex\hbox{$2$}}}}}}} \right.{\left( {F_{n,c}^{(\alpha )}\left( {{{(t - x)}^2};x} \right)} \right)^{{\raise0.7ex\hbox{$1$} \!\mathord{\left/
 {\vphantom {1 2}}\right.\kern-\nulldelimiterspace}
\!\lower0.7ex\hbox{$2$}}}}\\
\nonumber &\left. { + {{\left( {F_{n,m}^{c,\alpha}\left( {{{(t - x)}^2};x} \right)} \right)}^{{\raise0.7ex\hbox{$1$} \!\mathord{\left/
 {\vphantom {1 2}}\right.\kern-\nulldelimiterspace}
\!\lower0.7ex\hbox{$2$}}}}.{{\left( {F_{n,m}^{c,\alpha}\left( {{{(1 + ct)}^{ - \beta }};x} \right)} \right)}^{{\raise0.7ex\hbox{$1$} \!\mathord{\left/
 {\vphantom {1 2}}\right.\kern-\nulldelimiterspace}
\!\lower0.7ex\hbox{$2$}}}}} \right).
\end{align}
If $n$ is sufficiently large then we get
\begin{equation}\label{l8}
{\left( {F_{n,m}^{c,\alpha}\left( {{{(t - x)}^2};x} \right)} \right)^{{\raise0.7ex\hbox{$1$} \!\mathord{\left/
 {\vphantom {1 2}}\right.\kern-\nulldelimiterspace}
\!\lower0.7ex\hbox{$2$}}}} \le \sqrt {\frac{{2\alpha }}{n}} \phi (x),
\end{equation}
where $\phi (x) = \sqrt {x(1 + cx)}$.\\
For each $x\in[0,\infty)$, $F_{n,m}^{c,\alpha }\left( {{{(1 + ct)}^{ - \beta }};x} \right) \to {(1 + cx)^{ - \beta }}$ as $n \to \infty$. 
Thus for $\epsilon>0$, there exist $n_0\in \mathbb{N}$ such that
$$F_{n,m}^{c,\alpha }\left( {{{(1 + ct)}^{ - \beta }};x} \right) \le {(1 + cx)^{ - \beta }} +\varepsilon,\;\;\;\text{for all}\;n\ge n_0$$
By choosing $\varepsilon=(1+cx)^{-\beta}$ then, we get
\begin{equation}\label{l9}
F_{n,m}^{c,\alpha }\left( {{{(1 + ct)}^{ - \beta }};x} \right) \le 2{(1 + cx)^{ - \beta }},\;\;\;\text{for all}\; n \ge n_0.	
\end{equation}
From (\ref{l7}) to (\ref{l9}), we have
\begin{align}\label{l10}
\nonumber \left| {F_{n,m}^{c,\alpha }\left( {g(t);x} \right) - g(x)} \right| &\le {2^\beta }\left\| {{\phi ^\beta }g'} \right\|\sqrt {\frac{{2\alpha }}{n}} \phi (x)\left( {{\phi ^{ - \beta }}(x) + \sqrt 2 {x^{ - \frac{\beta }{2}}}{{(1 + cx)}^{ - \frac{\beta }{2}}}} \right)\\
&\le {2^{\beta  + \frac{1}{2}}}(1 + \sqrt 2 )\left\| {{\phi ^\beta }{g' }} \right\|\sqrt {\frac{\alpha }{n}} {\phi ^{1 - \beta }}(x).
\end{align}
We may write 
\begin{align}\label{l11}
\nonumber \left| {F_{n,m}^{c,\alpha }\left( {f(t);x} \right) - f(x)} \right| &\le \left| {F_{n,m}^{c,\alpha }\left( {f(t) - g(t);x} \right)} \right|\\
&\hspace{.5cm}\nonumber + \left| {F_{n,m}^{c,\alpha }\left( {g(t);x} \right) - g(x)} \right| + \left| {g(x) - f(x)} \right|\\
&\le 2\left\| {f - g} \right\| + \left| {F_{n,m}^{c,\alpha}\left( {g(t);x} \right) - g(x)} \right| .	
 \end{align}
From (\ref{l10}) to (\ref{l11}) and for sufficiently large $n$, we get
\begin{align}\label{l12}
\nonumber \left| {F_{n,m}^{c,\alpha }\left( {f(t);x} \right) - f(x)} \right| &\le 2\left\| {f - g} \right\| + {2^{\beta  + \frac{1}{2}}}(1 + \sqrt 2 )\sqrt {\frac{\alpha }{n}} {\phi ^{1 - \beta }}(x)\left\| {{\phi ^\beta }{g'}} \right\|\\
\nonumber &\le C_1\left\{ {\left\| {f - g} \right\| + \frac{{{\phi ^{1 - \beta }}(x)}}{{\sqrt n }}\left\| {{\phi ^\beta }{g'}} \right\|} \right\}	\\
&\le C{K_{{\phi ^\beta }}}\left( {f,\frac{{{\phi ^{1 - \beta }}(x)}}{{\sqrt n }}} \right),
\end{align}
where $C_1=max(2,{2^{\beta  + \frac{1}{2}}}(1 + \sqrt 2 )\sqrt \alpha)$ and $C=2C_1$. From (\ref{l01}) and (\ref{l12}), we get the required result.
\end{proof}
\end{theorem}

\section{Weighted Approximation}
For the estimation of the rate of convergence of the function $f \in C_2[0,\infty)$ by using the weighted modulus of continuity, which was introduced by Ispir and Y\"{u}ksal \cite{IN:2007} as follows:
\begin{equation}\label{we1}
	\Omega (f;\delta ) = \mathop {\sup }\limits_{x \in [0,\infty ),0 < \beta  < \delta } \frac{{f(x + \beta ) - f(x)}}{{1 + {{(x + \beta )}^2}}}.
\end{equation}
Many authors have already discussed weighted modulus of continuity for various linear positive operators. For more information (see \cite{AABK:18}, \cite{Deo:2007}, \cite{DPD:17}).\\

There are several properties of weighted modulus of continuity $\Omega (.;\delta )$ which are stated in following Lemma.
\begin{lem}\label{lmw1}
\cite{IN:2007} For $f\in C_2[0,\infty)$ the following properties hold:
\begin{enumerate}
\item[(i)] $\Omega (f;\delta)$ is monotonically increasing in $\delta ;$
\item[(ii)] $\mathop {\lim }\limits_{\delta  \to {0^ + }} \Omega (f;\delta ) = 0;$
\item[(iii)] For each $r \in \Bbb{N}$, $\Omega (f;r\delta ) \le r\Omega (f;\delta );$
\item[(iv)] For each $\lambda \in [0,\infty)$, $\Omega (f;\lambda \delta ) \le (\lambda  + 1)\Omega (f;\delta )$.
\end{enumerate}	
\end{lem}
\begin{theorem}
Let $f\in C_2[0,\infty)$, $\alpha>0$, for fixed $m$ and sufficiently large $n$ then, we have
\[\mathop {\sup }\limits_{x \in [0,\infty )} \frac{{\left| {F_{n,m}^{c,\alpha }(f;x) - f(x)} \right|}}{{{{(1 + x)}^{5/2}}}} \leqslant C\Omega \left( {f;\frac{1}{{\sqrt n }}} \right),\]
where $C$ is positive constant depends on $n$ and $f$.
\begin{proof}
By the definition of the weighted modulus of continuity and Lemma \ref{lmw1}, we have
\begin{align}\label{we2}
\nonumber\left| {f(t) - f(x)} \right| &\le \left( {1 + {{\left( {x + \left| {t - x} \right|} \right)}^2}} \right)\Omega \left( {f;\left| {t - x} \right|} \right)\\
&\le 2(1 + {x^2})\left( {1 + {{(t - x)}^2}} \right)\left( {1 + \frac{{\left| {t - x} \right|}}{\delta }} \right)\Omega (f;\delta).
\end{align} 
Applying $F_{n,m}^{c,\alpha }(.;x)$ on both side of (\ref{we2}), we can write 
\begin{align}\label{we3}
\nonumber \left| {F_{n,m}^{c,\alpha }(f;x) - f(x)} \right| &\le 
  \left[ {1 + F_{n,m}^{c,\alpha }({{(t - x)}^2};x)} \right. \hfill \\
  &\hspace{.5cm}\left. { + F_{n,m}^{c,\alpha }\left( {(1 + {{(t - x)}^2})\frac{{\left| {t - x} \right|}}{\delta };x} \right)} \right] .
\end{align}	
From Remark \ref{rem-1}, and using Cauchy-Schwarz inequality in the last term of (\ref{we3}), we have
\begin{align}\label{we4}
\nonumber F_{n,m}^{c,\alpha }\left( {(1 + {{(t - x)}^2})\frac{{\left| {t - x} \right|}}{\delta };x} \right) \le &{\rm{ }}\frac{1}{\delta }{\left( {\alpha \mu _{n,2}^c(x)} \right)^{1/2}}\\
  &+ \frac{1}{\delta }{\left( {\alpha \mu _{n,4}^c(x)} \right)^{1/2}}{\left( {\alpha \mu _{n,2}^c(x)} \right)^{1/2}}.\end{align}
Combining the estimate from (\ref{we2}) to (\ref{we4}) and taking
 $C = 2(1 + \sqrt {\alpha C}  + 2C)$ and $\delta  = \frac{1}{{\sqrt n }}$ then we get the required result.
\end{proof}
\end{theorem}

\section{function of bounded variation}
In this section, we study the rate of convergence of the B\'ezier variant of Gupta-Srivastava operators \eqref{L:2} in the class $DBV[0, \infty)$, the class of all absolutely continuous functions $f$ defined on $[0, \infty)$ having a derivative coinciding a.e. with a function of bounded variation on $[0, \infty)$. It can be seen that for $f \in DBV[0, \infty)$, we can write
\[f\left( x \right) = \int\limits_0^x {g\left( t \right)dt + f\left( 0 \right)},\]
where $g(t)$ is a function of bounded variation on each finite subinterval of $[0, \infty)$.
The operators $(\ref{L:2})$ can be rewritten in the following form:
\begin{equation}\label{E:3}
F_{n,m}^{c,\alpha}\left( {f(t);x} \right) = \int\limits_0^\infty  {M_{n,m,c}^{(\alpha )}(x,t)f(t)dt}, 
\end{equation}
where
\[M_{n,m,c}^{(\alpha )}(x,t) = \left\{ {n + (m + 1)c} \right\}\sum\limits_{k = 1}^\infty  {Q_{n + mc,k}^{(\alpha )}} (x;c){p_{n + (m + 2)c,k}}(t,c) + Q_{n + mc,0}^{(\alpha )}(x;c)\delta (t),\]
where $\delta(t)$ is Dirac delta function.
\begin{lem}\label{lem-4}
For a fixed $x\in [0,\infty)$ and $n$ is sufficient large then, we have
\begin{enumerate}
\item[(i)]$\zeta _{n,c}^{(\alpha )}(x;y) = \int\limits_0^y {M_{n,m,c}^{(\alpha )}(x;t)dt \le \frac{{2\alpha x(1 + cx)}}{{n{{(x - y)}^2}}}},\hspace{.2cm}0\le y \le x;$ 
\item[(ii)] $1 - \zeta _{n,c}^{(\alpha )}(x;z) = \int\limits_z^{\infty} {M_{n,m,c}^{(\alpha )}(x;t)dt \le \frac{{2\alpha x(1 + cx)}}{{n{{(z - x)}^2}}}},\hspace{.2cm}x\le z \le \infty.$	
\end{enumerate}
\end{lem}
\begin{proof}
From (\ref{E:3}), and using Remark \ref{rem-1} then, we have	
\begin{align*}
\zeta _{n,c}^{(\alpha )}(x;y) &\le \int\limits_0^y {M_{n,m,c}^{(\alpha )}(x;t){{\left( {\frac{{x - t}}{{x - y}}} \right)}^2}dt}\\
&\le \frac{\alpha }{{{{(x - y)}^2}}}{L_{n,c}}\left( {{{(e_1 - x)}^2};x} \right)\\
& \le \frac{{2\alpha x(1 + cx)}}{{n{{(x - y)}^2}}}.	
\end{align*}
We can prove the second part of Lemma in same way.
\end{proof}
\begin{theorem}\label{thm-3}
Let $f\in DBV[0,\infty)$ then for every $x\in[0,\infty)$ and $n$ is sufficiently large, we have 
\begin{align*}
\left| {F_{n,m}^{c,\alpha}\left( {f;x} \right) - f(x)} \right| \le & {\rm{ }}\frac{1}{{\alpha  + 1}}\left| {{f^\prime }(x + ) + \alpha {f^\prime }(x - )} \right|\sqrt {\frac{{2\alpha x(1 + cx)}}{n}}\\
&+ \frac{\alpha }{{\alpha  + 1}}\left| {{f^\prime }(x + ) - {f^\prime }(x - )} \right|\sqrt {\frac{{2\alpha x(1 + cx)}}{n}}\\
&+\frac{{2\alpha (1 + cx)}}{n}\sum\limits_{k = 1}^{\left[ {\sqrt n } \right]} {\mathop V\limits_{x - \frac{x}{k}}^x } f_x^{\prime}+\frac{x}{{\sqrt n }}\mathop V\limits_{x - \frac{x}{{\sqrt n }}}^x f_x^{\prime}\\
& +\frac{{2\alpha (1 + cx)}}{{nx}}\left| {f(2x) - f(x) - xf(x + )} \right|\\
\nonumber &+ \frac{{2\alpha x(1 + cx)}}{n}\sum\limits_{k = 1}^{\left[ {\sqrt n } \right]} {\mathop {{\rm{ }}V}\limits_x^{x + \frac{x}{k}} f_x^\prime }  + \frac{x}{{\sqrt n }}\mathop {{\rm{ }}V}\limits_x^{x + \frac{x}{{\sqrt n }}} (f_x^\prime )\\
&+M(\gamma ,r,x) + \frac{{2\alpha (1 + cx)}}{{nx}}\left| {f(x)} \right| + \sqrt {\frac{{2\alpha x(1 + cx)}}{n}} \left| {f(x + )} \right|.
\end{align*}
where $\mathop V\limits_a^b f(x)$ denotes the total variation of $f$ on $[a, b]$, ${f_x}$ is an auxiliary operator given by
\begin{equation}\label{v1}
{f_x}\left( t \right) = \left\{ \begin{array}{l}
f\left( t \right) - f\left( {x - } \right),\;0 \le t < x\\
0,\;\;\;\;\;\;\;\;\;\;\;\;\;\;\;\;\;\;\;\;\;\;t = x\\
f\left( t \right) - f\left( {x + } \right),\;x < t < \infty
\end{array} \right..
\end{equation}
\end{theorem}
\begin{proof}
From Remark $\ref{rem-2}$, $F_{n,m}^{c,\alpha}\left( {1;x} \right)=1$ and using the alternative form of the operators $(\ref{E:3})$ for each $x\in [0,\infty)$ then, we have
\begin{align}\label{v:2}
\nonumber F_{n,m}^{c,\alpha}\left( {f(t);x} \right) - f(x) &= \int\limits_0^\infty  {M_{n,m,c}^{(\alpha )}(x,t)\big(f(t) - f(x)\big) dt}\\
&= \int\limits_0^\infty  {M_{n,m,c}^{(\alpha )}(x,t)\left( {\int\limits_x^t {{f'}(u)du} } \right)dt} 	
\end{align}	
For each $f\in DBV[0,\infty)$ and from (\ref{v1}), we can write 
\begin{align}\label{v:3}
\nonumber {f'}(u)=& f_x'(u) + \frac{1}{{\alpha  + 1}}({f'}(x + ) + \alpha {f^{\prime}}(x - ))\\
\nonumber &+ \frac{1}{2}\left( {f'(x + ) + \alpha {f' }(x - )} \right)\left( {{\text{sgn}}(u - x) + \frac{{\alpha  - 1}}{{\alpha  + 1}}} \right)\\
&\times{\delta _x}(u)\left[ {{f' }(u) - \left( {{f'}(x + ) + {f'}(x - )} \right)} \right],
\end{align}
where 
\[{\delta _x}(u) = \left\{ \begin{array}{l}
1,\hspace{.5cm}u = x\\
0,\hspace{.5cm}u \ne x.
\end{array} \right.\]	
 From $(\ref{v:2})$ and $(\ref{v:3})$, we have
 \begin{align}\label{v:4}
\nonumber F_{n,m}^{c,\alpha}\left( {f(t);x} \right) - f(x) =& \int\limits_0^\infty  {M_{n,m,c}^{(\alpha )}(x,t)\int\limits_x^t {\left( {f_x'(u) + \frac{1}{{\alpha  + 1}}({f'}(x + ) + \alpha {\rm{ }}{f'}(x - ))} \right.} }\\
\nonumber &+ \frac{1}{2}({f'}(x + ) + \alpha {f'}(x - ))\left( {{\mathop{\rm sgn}} (u - x) + \frac{{\alpha  - 1}}{{\alpha  + 1}}} \right)\\
&\left. { \times {\delta _x}(u)[{f'}(u) - \frac{1}{2}({f'}(x + ) + {f'}(x - )]} \right)dudt. 	
 \end{align}
It is easy to say that 
\begin{equation}\label{v:5}
\int\limits_0^\infty  {M_{n,m,c}^{(\alpha )}(x,t)\int\limits_x^t {[{f' }(u) - \frac{1}{2}({f' }(x + ) + {f' }(x - )]{\delta _x}(u)} } dudt=0.
\end{equation}
Now
\begin{align}\label{v:6}
\nonumber {B_1}& = \int\limits_0^\infty  {M_{n,m,c}^{(\alpha )}(x,t)\int\limits_x^t {\frac{1}{{\alpha  + 1}}({f'}(x + ) + \alpha {f^{\prime}}(x - ))} } dudt.\\
\nonumber &= \frac{1}{{\alpha  + 1}}({f'}(x + ) + \alpha {f'}(x - ))\int\limits_0^\infty  {M_{n,m,c}^{(\alpha )}(x,t)(t - x)} dt	\\
&= \frac{1}{{\alpha  + 1}}({f'}(x + ) + \alpha {f'}(x - ))F_{n,c}^{(\alpha )}\left( {(t - x);x} \right),
\end{align}
and
\begin{align}\label{v:6}
\nonumber {B_2} &= \int\limits_0^\infty  {M_{n,m,c}^{(\alpha )}(x,t)\int\limits_x^t {\frac{1}{2}(f'(x + ) + \alpha {\rm{ }}f'(x - ))\left( {{\rm{sgn}}(u - x) + \frac{{\alpha  - 1}}{{\alpha  + 1}}} \right)} } dudt\\
\nonumber&= \frac{1}{2}({f'}(x + ) + \alpha {f' }(x - ))\left( { - \int\limits_0^x {M_{n,m,c}^{(\alpha )}(x,t)\int\limits_x^t {\left( {{\rm{sgn}}(u - x) + \frac{{\alpha  - 1}}{{\alpha  + 1}}} \right)} } dudt} \right.\\
\nonumber&\hspace{.8cm}+ \left. {\int\limits_x^\infty {M_{n,m,c}^{(\alpha )}(x,t)\int\limits_x^t {\left( {{\rm{sgn}}(u - x) + \frac{{\alpha  - 1}}{{\alpha  + 1}}} \right)} } } \right)\\
\nonumber &\le \frac{\alpha }{{\alpha  + 1}}({f' }(x + ) + \alpha {f'}(x - ))\int\limits_0^\infty  {M_{n,m,c}^{(\alpha )}(x,t)} \left| {t - x} \right|dt\\
 &\le \frac{\alpha }{{\alpha  + 1}}({f'}(x + ) + \alpha {f'}(x - )){\left( {F_{n,c}^{(\alpha )}\left( {{{({e_1} - x)}^2};x} \right)} \right)^{\frac{1}{2}}},
\end{align}
By using Remark $\ref{rem-1}$ and Lemma $\ref{lem-3}$, from $(\ref{v:4})-(\ref{v:6})$ then, we have
\begin{align}\label{v:7}
\nonumber F_{n,m}^{c,\alpha}\left( {f;x} \right) - f(x) \le & \left|{A_n^{(\alpha )}({f'};x) + B_n^{(\alpha )}({f'};x)} \right|\\
\nonumber & + \frac{{2\alpha }}{{\alpha  + 1}}\left| {{f' }(x + ) + \alpha {f'}(x - )} \right|\frac{{x(1 + cx)}}{n}\\
&+ \frac{\alpha }{{\alpha  + 1}}\left| {{f' }(x + ) - {f' }(x - )} \right|\sqrt {\frac{{2\alpha x(1 + cx)}}{n}},
\end{align}
where
\[A_n^{(\alpha )}({f'};x) = \int\limits_0^x {\left( {\int\limits_x^t {f_x'(u)du} } \right)M_{n,m,c}^{(\alpha )}(x,t)dt} ,\]
and \[B_n^{(\alpha )}({f'};x) = \int\limits_x^\infty  {\left( {\int\limits_x^t {f_x'(u)du} } \right)M_{n,m,c}^{(\alpha )}(x,t)dt} .\]
To estimate $A_n^{(\alpha )}({f^\prime };x)$, using integration by parts and applying  Lemma $\ref{lem-4}$ with $y = x - \frac{x}{{\sqrt n }}$, we obtain
\begin{align}
\nonumber A_n^{(\alpha )}({f'};x) &= \left| {\int\limits_0^x {\left( {\int\limits_x^t {f_x' (u)du} } \right){d_t}\zeta _{n,c}^{(\alpha )}(x;t)} } \right|\\	
\nonumber&= \left| {\int\limits_0^x {\zeta _{n,c}^{(\alpha )}(x;t)f_x'(t)dt} } \right|
\end{align}
\begin{align}\label{v:8}
\nonumber &\hspace{4.2cm}\le \int\limits_0^y {\left| {f_x'(t)} \right|\left| {\zeta _{n,c}^{(\alpha )}(x;t)} \right|dt + \int\limits_0^y {\left| {f_x'(t)} \right|\left| {\zeta _{n,c}^{(\alpha )}(x;t)} \right|dt} }\\
\nonumber &\hspace{4.2cm}\le \frac{{2\alpha x(1 + cx)}}{n}\int\limits_0^y {\mathop V\limits_t^x f_x'} {(x - t)^2}dt + \int\limits_y^x {\mathop V\limits_t^x f_x'} dt\\
&\hspace{4.2cm}\le \frac{{2\alpha x(1 + cx)}}{n}\int\limits_0^{x - \frac{x}{{\sqrt n }}} {\mathop V\limits_t^x f_x'} {(x - t)^2}dt + \frac{x}{{\sqrt n }}\mathop V\limits_{x - \frac{x}{{\sqrt n }}}^x f_x'.
\end{align}
Substituting $u = \frac{x}{{x - t}}$, we get
\begin{align}\label{v:9}
\nonumber \frac{{2\alpha x(1 + cx)}}{n}\int\limits_0^{x - \frac{x}{{\sqrt n }}} {\mathop V\limits_t^x f_x'} {(x - t)^2}dt& = \frac{{2\alpha x(1 + cx)}}{{nx}}\int\limits_1^{\sqrt n } {\mathop V\limits_{x - \frac{x}{u}}^x } f_x^{\prime}du\\
\nonumber &\le \frac{{2\alpha (1 + cx)}}{n}\sum\limits_{k = 1}^{\left[ {\sqrt n } \right]} {\int\limits_k^{k + 1} {\mathop V\limits_{x - \frac{x}{k}}^x } f_x^{\prime}du} \\
 & \le \frac{{2\alpha (1 + cx)}}{n}\sum\limits_{k = 1}^{\left[ {\sqrt n } \right]} {\mathop V\limits_{x - \frac{x}{k}}^x } f_x^{\prime}.	
\end{align}
From $(\ref{v:8})$ and $(\ref{v:9})$, we get
\begin{equation}\label{a:1}
A_n^{(\alpha )}({f'};x)=\frac{{2\alpha (1 + cx)}}{n}\sum\limits_{k = 1}^{\left[ {\sqrt n } \right]} {\mathop V\limits_{x - \frac{x}{k}}^x } f_x'+\frac{x}{{\sqrt n }}\mathop V\limits_{x - \frac{x}{{\sqrt n }}}^x f_x'.
\end{equation}
We can write
 \begin{align}\label{v:10}
 \nonumber B_n^{(\alpha )}({f^\prime };x) \le &\left| {\int\limits_x^{2x} {\left( {\int\limits_x^t {f_x^\prime (u)du} } \right){d_t}(1 - \zeta _{n,c}^{(\alpha )}(x;t))} } \right|\\
\nonumber &\hspace{.2cm}+ \left| {\int\limits_{2x}^\infty  {\left( {\int\limits_x^t {f_x^\prime (u)du} } \right){d_t}M_{n,m,c}^{(\alpha )}(x,t)} } \right|.
\end{align}
From the second part of the Lemma $\ref{lem-4}$, we get\\
$$M_{n,m,c}^{(\alpha )}(x,t) = {d_t}((1 - \zeta _{n,c}^{(\alpha )}(x;t))\hspace{.5cm}for\hspace{.2cm} t>x.$$
Hence
\[B_n^{(\alpha )}({f^\prime };x) = B_{n,1}^{(\alpha )}({f^\prime };x) + B_{n,2}^{(\alpha )}({f^\prime };x),\]
where
\[B_{n,1}^{(\alpha )}({f^\prime };x) = \left| {\int\limits_x^{2x} {\left( {\int\limits_x^t {f_x^\prime (u)du} } \right){d_t}(1 - \zeta _{n,c}^{(\alpha )}(x;t))} } \right|,\]
and
\[B_{n,2}^{(\alpha )}({f^\prime };x) = \left| {\int\limits_{2x}^\infty  {\left( {\int\limits_x^t {f_x^\prime (u)du} } \right){d_t}M_{n,m,c}^{(\alpha )}(x,t)} } \right|.\] 
Using integration by parts, applying Lemma $\ref{lem-4}$, ${1 - \zeta _{n,c}^{(\alpha )}(x;t)}\le 1$ and taking $t = x + \frac{x}{u}$ successively,
\begin{align}\label{v:11}
\nonumber B_{n,1}^{(\alpha )}({f^\prime };x) = &\left| {\int\limits_x^{2x} {f_x^\prime (u)du(1 - \zeta _{n,c}^{(\alpha )}(x;2x)) - \int\limits_x^{2x} {f_x^\prime (t)(1 - \zeta _{n,c}^{(\alpha )}(x;t))dt} } } \right|\\
\nonumber \le&\left| {\int\limits_x^{2x} {{(f_{}^\prime (u) - f_{}^\prime (x + ))}du} } \right|\left| {1 - \zeta _{n,c}^{(\alpha )}(x;2x)} \right| + \left| {\int\limits_x^{2x} {f_x^\prime (t)(1 - \zeta _{n,c}^{(\alpha )}(x;t))dt} } \right|\\
\nonumber\le& \frac{{2\alpha (1 + cx)}}{{nx}}\left| {f(2x) - f(x) - xf(x + )} \right|\\
 \nonumber&+ \frac{{2\alpha x(1 + cx)}}{n}\int\limits_{x + \frac{x}{{\sqrt n }}}^{2x} {\frac{{\mathop V\limits_x^t f_x^\prime }}{{{{(t - x)}^2}}}} dt + \int\limits_x^{x + \frac{x}{{\sqrt n }}} {\mathop V\limits_x^t f_x^\prime dt} \\
 \nonumber\le& \frac{{2\alpha (1 + cx)}}{{nx}}\left| {f(2x) - f(x) - xf(x + )} \right|\\
 &+ \frac{{2\alpha x(1 + cx)}}{n}\sum\limits_{k = 1}^{\left[ {\sqrt n } \right]} {\mathop V\limits_x^{x + \frac{x}{k}} f_x^\prime }  + \frac{x}{{\sqrt n }}\mathop V\limits_x^{x + \frac{x}{{\sqrt n }}} (f_x^\prime ).
\end{align}
Using Remark $\ref{rem-1}$ then, we have 
\begin{align}\label{v:12}
\nonumber B_{n,2}^{(\alpha )}({f^\prime };x) = &\left| {\int\limits_{2x}^\infty  {\left( {\int\limits_x^t {({f^{\prime}}(u) - {f^{\prime}}(x + ))du} } \right)M_{n,m,c}^{(\alpha )}(x,t)dt} } \right|\\
 \nonumber \le &\int\limits_{2x}^\infty  {\left| {f(t) - f(x)} \right|} M_{n,m,c}^{(\alpha )}(x,t)dt + \int\limits_{2x}^\infty  {\left| {t - x} \right|} \left| {f(x + )} \right|M_{n,m,c}^{(\alpha )}(x,t)dt\\
 \nonumber \le& \left| {\int\limits_{2x}^\infty  {f(t)M_{n,m,c}^{(\alpha )}(x,t)dt} } \right| + \left| {f(x)} \right|\left| {\int\limits_{2x}^\infty  {M_{n,m,c}^{(\alpha )}(x,t)dt} } \right|\\
\nonumber  &+ \left| {f(x + )} \right|{\left( {\int\limits_{2x}^\infty  {{{(e_1 - x)}^2}M_{n,m,c}^{(\alpha )}(x,t)dt} } \right)^{\frac{1}{2}}}\\
\nonumber\le& M\int\limits_{2x}^\infty  {{t^\gamma }M_{n,m,c}^{(\alpha )}(x,t)dt}  + \left| {f(x)} \right|\left| {\int\limits_{2x}^\infty  {M_{n,m,c}^{(\alpha )}(x,t)dt} } \right|\\
 &+ \sqrt {\frac{{2\alpha x(1 + cx)}}{n}} \left| {f(x + )} \right|.
\end{align}
For $t\ge 2x$, we get $t\le 2(t-x)$ and $x\le t-x$, applying H$\ddot{\text{o}}$lder's inequality, we have
\begin{align}\label{v:13}
\nonumber B_{n,2}^{(\alpha )}({f^{\prime} };x) \le & M{2^\gamma }{\left( {\int\limits_{2x}^\infty  {{{(e_1 - x)}^{2r}}M_{n,m,c}^{(\alpha )}(x,t)dt} } \right)^{\frac{\gamma }{{2r}}}}\\
\nonumber & + \frac{{2\alpha (1 + cx)}}{{nx}}\left| {f(x)} \right| + \sqrt {\frac{{2\alpha x(1 + cx)}}{n}} \left| {f(x + )} \right|\\
 = & M(\gamma,c,r,x) + \frac{{2\alpha (1 + cx)}}{{nx}}\left| {f(x)} \right| + \sqrt {\frac{{2\alpha x(1 + cx)}}{n}} \left| {f(x + )} \right|.
\end{align}
From (\ref{v:11}) and (\ref{v:13}), we get
\begin{align}\label{v:14}
\nonumber B_n^{(\alpha )}({f^\prime };x) =& \frac{{2\alpha (1 + cx)}}{{nx}}\left| {f(2x) - f(x) - xf(x + )} \right|\\
\nonumber &+ \frac{{2\alpha x(1 + cx)}}{n}\sum\limits_{k = 1}^{\left[ {\sqrt n } \right]} {\mathop {{\rm{ }}V}\limits_x^{x + \frac{x}{k}} f_x^\prime }  + \frac{x}{{\sqrt n }}\mathop {{\rm{ }}V}\limits_x^{x + \frac{x}{{\sqrt n }}} (f_x^\prime )\\
&+M(\gamma,c,r,x) + \frac{{2\alpha (1 + cx)}}{{nx}}\left| {f(x)} \right| + \sqrt {\frac{{2\alpha x(1 + cx)}}{n}} \left| {f(x + )} \right|.
\end{align}
From $(\ref{v:7})$, $(\ref{a:1})$ and $(\ref{v:14})$, we get our desired result.
\end{proof}

\end{document}